\documentclass[a4paper,11pt]{amsart}
\usepackage{amsmath,amssymb,amsthm}
\usepackage{graphicx,color,dsfont}

\numberwithin{equation}{section}

\newtheorem{Theorem}{Theorem}[section]

\newtheorem{Corollary}[Theorem]{Corollary}
\newtheorem{Lemma}[Theorem]{Lemma}
\newtheorem{Proposition}[Theorem]{Proposition}
\newtheorem{Remark}{Remark}[section]

\begin{document}

\title
[Blow-up for self-interacting fractional GL eq]
{Blow-up for self-interacting fractional Ginzburg-Landau equation}

\author[K. Fujiwara]{Kazumasa Fujiwara}

\address{%
Department of Pure and Applied Physics \\ Waseda University \\
3-4-1, Okubo, Shinjuku-ku, Tokyo 169-8555 \\ Japan}

\email{k-fujiwara@asagi.waseda.jp}

\thanks{
The first author was partly supported by the Japan Society for the Promotion of Science,
Grant-in-Aid for JSPS Fellows no 16J30008
and Top Global University Project of Waseda University.
}

\author[V. Georgiev]{Vladimir Georgiev}

\address{%
Department of Mathematics \\
University of Pisa \\
Largo Bruno Pontecorvo 5
I - 56127 Pisa
\\ Italy  \\ and \\
 Faculty of Science and Engineering \\ Waseda University \\
 3-4-1, Okubo, Shinjuku-ku, Tokyo 169-8555 \\
Japan}

\email{georgiev@dm.unipi.it}

\thanks{ The second author was supported in part by  INDAM, GNAMPA - Gruppo Nazionale per l'Analisi Matematica, la Probabilita e le loro Applicazioni, by Institute of Mathematics and Informatics, Bulgarian Academy of Sciences and Top Global University Project, Waseda University.}

\author[T. Ozawa]{Tohru Ozawa}

\address{%
Department of Applied Physics \\ Waseda University \\
3-4-1, Okubo, Shinjuku-ku, Tokyo 169-8555 \\ Japan}

\email{txozawa@waseda.jp}
\thanks{The third author was supported by
Grant-in-Aid for Scientific Research (A) Number 26247014.}

\begin{abstract}
The blow-up of solutions for the Cauchy problem of fractional Ginzburg-Landau equation
with non-positive nonlinearity is shown by an ODE argument.
Moreover, in one dimensional case,
the optimal lifespan estimate for size of initial data is obtained.
\end{abstract}

\maketitle

\renewcommand{\thefootnote}{\fnsymbol{footnote}}

\setcounter{page}{1}

\section{Introduction}
The classical complex Ginzburg-Landau  (CGL) equation takes the form
	\begin{equation}
	\partial_t \psi =-(\alpha+\mathrm i \beta) \Delta \psi + F(\psi, \overline{\psi}),
	\label{eq:1.1}
	\end{equation}
where $\alpha, \beta$ are real parameters.
The standard CGL equation has a self-interaction term $F$ of the form
	\[
	F(\psi,\overline{\psi})
	= - \sum_{j=1}^K (\alpha_j + \mathrm i \gamma_j) \psi |\psi|^{p_j-1},
	\]
where $\alpha_j, \beta_j$ are real parameters.
We refer to \cite{vS} for a review on this subject.
Using the representation $\psi(t,x) = u_1(t,x) + \mathrm i u_2(t,x)$,
where $u_1, u_2$ are real-valued functions,
we see that the equation \eqref{eq:1.1} can be rewritten
in the form of a system of reaction diffusion equations
	\[
	\partial_t U = A \Delta U = F(U),
	\]
where
	\[
	U(t,x) = \left(
	\begin{array}{c}
	u_1(t,x) \\
	u_2(t,x) \\
	\end{array}
	\right), \quad
	A = \left(
	\begin{array}{cc}
	-\alpha & \beta \\
	-\beta & -\alpha \\
	\end{array}
	\right).
	\]
The limiting case $\alpha \to 0, \alpha_j \to 0$
leads to the nonlinear Schr\"odinger equation (NLS)
	\begin{equation}\label{eq:1.2}
	\partial_t \psi
	=-\mathrm i \beta  \Delta \psi - \sum_{j=1}^N \mathrm i \gamma_j  \psi |\psi|^{p_j-1}.
	\end{equation}
The oscillation synchronization of phenomena modeled by Kuramoto equations
(see \cite{Kura1}) lead to a system of ODE having a similar qualitative behavior
	\begin{equation}\label{eq:1.3}
	\partial_t \psi_k =-\mathrm i H_k \psi_k + F_k(\Psi, \overline{\Psi}),\ \ k=1, \cdots, N.
	\end{equation}
The nonlinear terms $F_k$ in the system obey the property
	\[
	{\rm Im} \ \left( F_k(\Psi, \overline{\Psi}) \ \ \overline{\psi_k}\right) =0,
	\ k=1, \cdots, N.
	\]
This system simulates the behavior of $N$ oscillators,
so that $\Psi=(\psi_1,\cdots, \psi_N), $ with $\psi_j$ being complex-valued functions.
The nonlinearities in \eqref{eq:1.3} are chosen
so that the evolution flow associated to the Kuramoto system leaves the manifold
	\[
	\mathcal{M}
	= \underbrace{\mathbb{S}^1 \times \cdots \times \mathbb{S}^1}_{N \ \mbox{times}},
	\]
invariant.

The derivation of the Kuramoto system in \cite{Kura1}
is based on complex Landau-Ginzburg equation
(see equation (2.4.15) in \cite{Kura1})
	\[
	\partial_t \Psi
	= \mathrm i \mathcal{H}\Psi - (\alpha+\mathrm i \beta) \Delta \Psi
	- (\alpha_1+\mathrm i \beta_1) \Psi |\Psi|^2,
	\]
where $ \Psi=(\psi_1,\cdots,\psi_N)^t$,
$\mathcal{H}$ is a diagonal matrix with real entries.
If $\beta$ and $\beta_1$ become very large,
then we have an equation very close to Schr\"odinger self-interacting system \eqref{eq:1.2}.
As it was pointed out (p. 20, \cite{Kura1}),
a chemical turbulence of a diffusion-induced type
are possible only for regions intermediate between the two extreme cases,
where $\beta$ and $\beta_1$ are very small or very large.

Turning back to CGL equation and comparing \eqref{eq:1.1} with Kuramoto system,
we see that it is natural to take $\alpha \to 0,$ $\beta_j \to 0$
so that we have the following simplified CGL equation
	\[
	\partial_t \psi =-\mathrm i \beta \Delta \psi - \alpha_1 \psi |\psi|^{p-1}.
	\]
A similar system was discussed in \cite{ACGM} with nonlinearity typical for the Kuramoto system.

The fractional dynamics seems more adapted to synchronization models
due to the considerations in \cite{TZ},
therefore we can consider the following fractional Ginzburg-Landau equations
	\[
	\partial_t \psi =-\mathrm i \sqrt{- \Delta} \psi \pm  \psi |\psi|^{p-1}.
	\]
The study of the attractive case
	\[
	\partial_t \psi =-\mathrm i |D| \psi - \psi |\psi|^{p-1},\quad
	|D|=\sqrt{- \Delta},
	\]
is initiated in \cite{FGO1},
where the well-posedness is established for the cases $1 \leq n \leq 3.$

In this article, we study the repulsive case
\begin{align}
	\begin{cases}
	 \partial_t u = -\mathrm i |D| u +
|u|^{p-1} u,
	& t \in \lbrack 0, T), \quad x \in \mathbb R^n,\\
	u(0,x) = u_0(x),
	& x \in \mathbb R^n,
	\end{cases}
	\label{eq:1.4}
	 \end{align}
where  $n \geq 1$, and $p > 1$.
Our main goal is to obtain a blow-up result under the assumption
that initial data are in $H^s(\mathbb R^n)$ with $s > n/2$,
where $H^s(\mathbb R^n)$ is the usual Sobolev space
defined by $(1-\Delta)^{-s/2} L^2(\mathbb R^n)$.

We denote $\langle x \rangle = (1+|x|^2)^{1/2}$.
We abbreviate $L^q(\mathbb R^n)$ to $L^q$ and
$\| \cdot \|_{L^q(\mathbb R^n)}$ to $\| \cdot \|_q$ for any $q$.
We also denote by $\|T\|$ the operator norm of bounded operator $T:L^2 \to L^2$.

The following statements are the main results of this article.

\begin{Proposition}
\label{Proposition:1.1}
Let $h$ be a Lipschitz function satisfying $\frac 1 h \in L^\infty$ and
	\begin{align}
	\bigg\| \frac{1}{h(\cdot)}
	\int_{\mathbb R^n} \langle \cdot - y \rangle^{-n-1} h(y) f(y) dy \bigg\|_2
	\leq C \|f\|_2.
	\label{eq:1.5}
	\end{align}
Let $u_0 \in h L^2$ satisfy
	\begin{align}
	\| \frac{1}{h} u_0 \|_{2}
	\geq \|\frac{1}{h} [D,h]\|^{\frac 1 {p-1}} \| \frac{1}{h} \|_{2}.
	\label{eq:1.6}
	\end{align}
If there is a solution $u \in C([0,T);h L^2)$ for \eqref{eq:1.4}, then
	\begin{align}
	&\| \frac{1}{h} u(t) \|_{2}
	\label{eq:1.7}\\
	&\geq e^{- 2 \|\frac{1}{h} [D,h]\| t}
	\Big( \| \frac{1}{h} u_0 \|_{2}^{-p+1}
	+ \|\frac{1}{h} [D,h]\|^{-1} \| \frac 1 h \|_{2}^{-p+1}
	\big\{ e^{- \|\frac{1}{h} [D,h]\| (p-1) t} - 1 \big\}
	\Big)^{-\frac{1}{p-1}}.
	\nonumber
	\end{align}
Therefore, the lifespan is estimated by
	\begin{align}
	T \leq - \frac 2 {p-1} \|\frac{1}{h} [D,h]\|^{-1}
	\log \bigg( 1 - \|\frac{1}{h} [D,h]\| \| \frac 1 h \|_{2}^{p-1}
	\| \frac{1}{h} u_0 \|_{2}^{-p+1} \bigg).
	\label{eq:1.8}
	\end{align}
\end{Proposition}

\bigskip

We remark that $\langle \cdot \rangle$ is a typical example of weight functions $h$
for Proposition \ref{Proposition:1.1}.
Proposition \ref{Proposition:1.1} is a blow-up result for a kind of large data of $h L^2$.
However, in a subcritical case where $p < p_F= 1 + 2/n$,
solutions blow up even for small $L^2$ initial data.

\begin{Corollary}
\label{Corollary:1.2}
Let $u_0 \in L^2(\mathbb R^n) \backslash \{0\}$ and $1 < p < p_F$.
Then the corresponding solution in $C([0,T);L^2)$ blows up at a finite positive time.
\end{Corollary}

\begin{Remark}
If we choose $h(x) = \langle x \rangle$,
the statement of our main result guarantees the blow-up of the momentum
	\[
	Q_{-1}(t) = \int_{\mathbb R^n} \langle x \rangle^{-1} |u(t,x)|^2 dx.
	\]
for the solution to the fractional CGL equation
	\[
	\partial_t u = -\mathrm i |D| u + |u|^{p-1} u
	\]
in \eqref{eq:1.4}.
The blow-up mechanism is based on the differential inequality
	\begin{equation}\label{eq:1.9}
	Q^\prime_{-1}(t) \geq  C_0 \left( Q_{-1} (t) \right)^{(p+1)/2}  - C_1 Q_{-1} (t),
	\ \ C_0, C_1 >0.
	\end{equation}
Comparing the fractional CGL equation with the classical NLS
	\[
	\partial_t u = \mathrm i \Delta u + \mathrm i |u|^{p-1} u,
	\]
we see that introducing  the momentum
	\[
	Q_2(t) = \int_{\mathbb R^n} |x|^2 |u(t,x)|^2 dx,
	\]
and using a Virial identity one can show that
	\[
	Q_2^{\prime \prime}(t) \sim E(u)(t),
	\]
where $E(u)(t) = \| D u(t)\|^2_{L^2} - c \|u(t)\|^{p+1}_{L^{p+1}}$
and $c>0$ is an appropriate constant.
Therefore, the blow-up mechanism for NLS is based on the estimate
	\[
	E(u)(t) \leq -\delta, \ \delta >0,
	\]
that implies differential inequality
	\[
	Q_2^{\prime \prime}(t) \leq - \delta
	\]
and the last inequality can not be satisfied for the whole interval $t \in (0,\infty)$
since $Q_2(t)$ is a positive quantity.
\end{Remark}

Moreover, for large $R$,
if $u_0$ is given by $Rf$ with $f \in h L^2$ and $h$ satisfying \eqref{eq:1.5},
then \eqref{eq:1.8} means $T \leq C R^{-p+1}$.
In one dimensional case,
this upper bound is shown to be sharp for $f \in (h L^2 \cap H^1)(\mathbb R)$.

\begin{Proposition}
\label{Proposition:1.3}
Let $u_0 = R f$ with $R > 0$ sufficiently large and $f \in H^1(\mathbb R)$.
Then there exists an $H^1(\mathbb R)$ solution for $u_0$
for which its lifespan is estimated by $T \geq C R^{-p+1}$
with some positive constant $C$.
\end{Proposition}


\section{Preliminary}
In this section, we recall the blow-up solutions for an ODE
which gives the mechanism of blow-up for weighted $L^2$ norm of solutions.
We also study the condition for weight functions of Proposition \ref{Proposition:1.1}.

\subsection{Blow-up solutions for an ODE}
\begin{Lemma}
\label{Lemma:2.1}
Let $C_1, C_2 > 0$ and $q>1$. If $f \in C^1([0,T); \mathbb R)$ satisfies $f(0) > 0$ and
	\[
	f' + C_1 f = C_2 f^q \quad \mbox{on $[0,T)$ for some $T>0$},
	\]
then
	\[
	f(t)
	= e^{-C_1 t}
	\bigg( f(0)^{-(q-1)} + C_1^{-1} C_2 e^{-C_1(q-1)t}
	- C_1^{-1} C_2 \bigg)^{-\frac{1}{q-1}}.
	\]
Moreover, if $f(0) > C_1^{\frac 1 {q-1}} C_2^{- \frac 1 {q-1}}$,
then $T < - \frac 1 {C_1 (q-1)} \log ( 1 - C_1 C_2^{-1} f(0)^{-q+1} )$.
\end{Lemma}

\begin{proof}
Let $f = e^{-C_1 t} g$. Then
	\[
	g' = C_2 e^{-C_1(q-1)t} g^q.
	\] Therefore,
	\[
	\frac{1}{1-q} \bigg( g^{1-q}(t) - g^{1-q}(0) \bigg)
	=
\frac{C_2}{C_1(1-q)} ( e^{-C_1(q-1)t} - 1).
	\]
\end{proof}

\subsection{Condition for weight function}
\begin{Lemma}[Coiffman - Meyer]
\label{Lemma:2.2}
Let $p \in C^\infty(\mathbb R^{2n})$ satisfy the estimates
	\[
	|D_x^\beta D_\xi^\alpha p(x,\xi) |
	\leq C_{\alpha,\beta} \langle \xi \rangle^{1-|\alpha|}
	\]
for all multi-indices $\alpha$ and $\beta$.
Then
for any Lipschitz function $h$,
	\[
	\| [p(x,D),h] f\|_2
	\leq C \| h \|_{\mathrm{Lip}} \| f\|_2.
	\]
\end{Lemma}

\begin{Lemma}
Let $\phi \in C_0^\infty([0,\infty); \mathbb R)$ satisfy
	\[
	\phi(\rho) =
	\begin{cases}
	1 & \mathrm{if} \quad 0 \leq \rho \leq 1,\\
	0 & \mathrm{if} \quad \rho \geq 2.
	\end{cases}
	\]
	Then
	\[
	\bigg| \int_{\mathbb R^n} \phi(|\xi|) |\xi| e^{\mathrm i x \cdot \xi} d \xi \bigg|
	\leq C \langle x \rangle^{-n-1}.
	\]
\end{Lemma}

\begin{proof}
It suffices to consider the case where $|x|$ is sufficiently large.
Let $\psi \in C_0^\infty([0,\infty); \mathbb R)$ satisfy
	\[
	\psi(\rho) =
	\begin{cases}
	1 & \mathrm{if} \quad 0 \leq \rho \leq 1,\\
	0 & \mathrm{if} \quad \rho \geq 2.
	\end{cases}
	\]
Let $e_1 = (1,0,\cdots,0)$.
Let $\xi_1 = \xi \cdot e_1$ and $\xi' = \xi - \xi_1 e_1$.
Assume $x = |x| e_1$.
Then
	\[
	\int_{\mathbb R^n} \phi(|\xi|) |\xi| e^{\mathrm i x \cdot \xi} d \xi
	= \int_{\mathbb R^n} \phi(|\xi|) |\xi| e^{\mathrm i |x| \xi_1} d \xi.
	\]
By integrating by parts $k$ times,
	\begin{align*}
	\int_{\mathbb R^n} \phi(|\xi|) |\xi| e^{\mathrm i |x| \xi_1} d \xi
	&= (- \mathrm i|x|)^{-k} \!
	\int_{\mathbb R^n} \partial_1^k (\phi(|\xi|) |\xi|) e^{ \mathrm i |x| \xi_1} d \xi\\
	&= (-\mathrm i|x|)^{-k} \!
	\int_{\mathbb R^n} \partial_1^k(|\xi|) \phi(|\xi|) e^{\mathrm i |x| \xi_1}
	+ R_k(\xi) d \xi,
	\end{align*}
where $R_{n+1} \in L^1(\mathbb R^n)$.
Here $\partial_1^k |\xi|$ is estimated by $C |\xi|^{1-k}$.
Moreover,
	\begin{align*}
	&\bigg| \int_{\mathbb R^{n}}
	\partial_1^n(|\xi|) \phi(|\xi|) e^{ \mathrm i |x| \xi_1} d \xi \bigg|\\
	&= |x|^{-1} \bigg| \int_{\mathbb R^{n-1}} \int_{\mathbb R}
	\partial_1 \{ \partial_1^n(|\xi|) \phi(|\xi|) \} ( e^{\mathrm i |x| \xi_1} - 1)
	d \xi_1 d \xi'
	\bigg|\\
	&\leq C \bigg| \int_{\mathbb R^n} \psi(|x| |\xi|)
	|\xi|^{1-n} |\phi(|\xi|)| + |\partial_1 \phi(|\xi|)| \} d \xi \bigg|\\
	&+ C |x|^{-1} \bigg| \int_{\mathbb R^n}
	(1-\psi(|x| |\xi|))
	|\xi|^{-n} \{ |\phi(|\xi|)| + |\partial_1 \phi(|\xi|)| \} d \xi \bigg|
	\end{align*}
since $|e^{\mathrm i |x| \xi_1} -1 | \leq |x| |\xi|$.
The first integral is estimated by
	\[
	C \int_0^{2 |x|^{-1}} |\phi(\rho)| + |\phi'(\rho)| d
\rho
	\leq C \| \phi \|_{C^1(0,2)} |x|^{-1}.
	\] By letting $\Psi(\rho) = \int_0^\rho |1 - \psi(\rho')| d \rho'$ and integrating by parts once again, the
second integral is estimated by
	\[
	C |x|^{-2} \int_{|x|^{-1}}^2 \rho^{-2}
	\| \phi \|_{C^2(0,2)} \Psi(|x|\rho) d \rho
	\leq C \| \phi \|_{C^2(0,2)}
|x|^{-1}.
	\] This proves the lemma.
\end{proof}

\begin{Lemma}
\label{Lemma:2.4}
Let $h$ be a Lipschitz function on $\mathbb R^n$ satisfying the estimate
	\[
	\bigg\| \frac 1 {h(\cdot)}
	\int_{\mathbb R^n} \langle \cdot - y \rangle^{-n-1} h(y) f(y) dy \bigg\|_2
	\leq C \|f\|_2
	\]
for any $f \in L^2$.
Then $\frac{1}{h} [D,h]$ is a bounded operator from $L^2$ to $L^2$.
\end{Lemma}

\begin{proof}
Let $\phi$ be a smooth function on $[0,\infty)$
satisfying that $\phi(\xi) = 1$ if $|\xi| \leq 1$ and $\phi(\xi) = 0$ if $|\xi| \geq 2$.
Let $\phi(D) f = \mathfrak F^{-1} \phi \hat f$.
We divide the proof into the following two estimate:
$\| \frac{1}{h} \phi(D) (D h f)\|_2 \leq \|f\|_2 $
and $\| \frac{1}{h} [ (1-\phi(D))D, h] f \|_2 \leq \|f\|_2 $.

At first,
	\[
	\| \frac{1}{h} \phi(D) (D h f)\|_2
	\leq C \bigg\| \frac 1 {h(\cdot)}
	\int_{\mathbb R^n} \langle \cdot - y \rangle^{-n-1} h(y) f(y) dy \bigg\|_2
	\leq C \|f\|_2,
	\]
since
	\[
	| \mathfrak F^{-1} ( |\cdot| \phi )| \leq C \langle x \rangle^{-n-1}.
	\]

Secondly, $(1-\phi(|\xi|))|\xi|$ satisfies the condition of Lemma \ref{Lemma:2.2}.
So the second estimate follows from Lemma \ref{Lemma:2.2}.
\end{proof}

\begin{Remark}
$h(x) = \langle x \rangle$ satisfies the condition of Lemma \ref{Lemma:2.4}.
Actually $h$ is Lipshitz and by using triangle inequality,
	\begin{align*}
	&\bigg\| \langle x \rangle^{-1}
	\int_{\mathbb R^n} \langle x-y \rangle^{-n-1} \langle y \rangle f (y) dy
	\bigg\|_2\\
	&\leq \bigg\|
	\int_{\mathbb R^n} \langle x-y \rangle^{-n-1} f (y) dy
	\bigg\|_2
	+ \bigg\| \langle x \rangle^{-1}
	\int_{\mathbb R^n} \langle x-y \rangle^{-n} f (y) dy
	\bigg\|_2\\
	&\leq ( \| \langle \cdot \rangle^{-n-1} \|_1
	+ \| \langle \cdot \rangle^{-1} \|_q
	\| \langle \cdot \rangle^{-n} \|_{q'} )
	\| f\|_2,
	\end{align*}
where $n < q < \infty$.
\end{Remark}

\begin{Corollary}
\label{Corollary:2.5}
Let $h$ satisfy the condition of Lemma \ref{Lemma:2.4} and let $h_R$ be $h_R = h(\cdot/R)$.
Then
	\[
	\| \frac 1 h_R [D, h_R] \|
	\leq R^{-1} \|\frac{1}{h} [D,h] \|.
	\]
\end{Corollary}

\begin{proof}
	\begin{align*}
	\frac 1 {h_R(x)} [D, h_R] f(x)
	&= \frac{1}{h(\frac x R)} \int_{\mathbb R^n} e^{\mathrm i (x-y) \cdot \xi}
	|\xi| \{ h(\frac y R) - h(\frac x R) \}
	f(y) d \xi dy\\
	&= R^n \frac{1}{h (\frac x R)} \int_{\mathbb R^n} e^{\mathrm i (\frac x R -y) \cdot R \xi}
	|\xi| \{ h(y) - h(\frac x R) \} f(Ry) d \xi dy\\
&= R^{-1} \frac{1}{h (\frac x R)} \int_{\mathbb R^n} e^{\mathrm i (\frac x R -y) \cdot \xi}
	|\xi| \{ h(y) - h(\frac x R) \} f(Ry) d \xi dy\\
	&= R^{-1}
\frac{1}{h(\frac x R)} [D,h] f_{R^{-1}}(\frac x R).
	\end{align*}
This implies
	\[
	\| \frac 1 {h_R} [D, h_R] f \|_2
	= R^{-1+n/2} \| \frac{1}{h}
[D,h] f_{R^{-1}} \|_2
	\leq R^{-1} \|\frac{1}{h} [D,h] \| \| f \|_2.
	\]
\end{proof}

\section{Proof}
\subsection{Proof of Proposition \ref{Proposition:1.1}}
Let $u(t,x) = h(x) v(t,x)$.
Then
	\begin{align}
	\mathrm i \partial_t v + D v + \frac{1}{h} [D,h] v
	= \mathrm i h^{p-1} |v|^{p-1} v.
	\label{eq:3.1}
	\end{align}
Multiplying both hand sides of \eqref{eq:3.1} by $\overline v$,
integrating over $\mathbb R^n$,
and taking the imaginary part of the resulting integrals,
we obtain
	\begin{align*}
	\frac 1 2 \frac d {dt} \|v(t)\|_{2}^2
	&= \int_{\mathbb R^n} h(x)^{p-1} |v(t,x)|^{p+1} dx
	- \mathrm{Im} \int_{\mathbb R^n}
	\overline{v(t,x)} \frac{1}{h(x)} [D,h] v(t,x) dx\\
	&\geq \int_{\mathbb R^n} h(x)^{p-1} |v(t,x)|^{p+1} dx
	- \|\frac{1}{h}[D,h]\| \| v(t) \|_{2}^2\\
	&\geq \| \frac 1 h \|_{2}^{-p+1} \| v(t) \|_{2}^{p+1}
	- \|\frac{1}{h}[D,h]\| \| v(t) \|_{2}^2,
	\end{align*}
where we used the following estimate:
	\[
	\| v(t) \|_{2} \leq \| \frac 1 {h^{\frac{p-1}{p+1}} } \|_{\frac{2(p+1)}{p-1}}
	\| h^{\frac{p-1}{p+1}} v(t) \|_{p+1}.
	\]
Then \eqref{eq:1.7} follows from Lemma \ref{Lemma:2.1} with $q=(p+1)/2$.

\subsection{Proof of Corollary \ref{Corollary:1.2}}
Let $h_R (x) = \langle x / R \rangle$ with $R > 0$.
Then $\frac 1 {h_R} u_0 \to u_0$ in $L^2$ as $R \to \infty$.
Moreover, $\| \frac 1 {h_R} [ D, h_R] \| \sim R^{-1}$,
and $\| \frac 1 {h_R} \|_{2} \sim R^{n/2}$.
Therefore
	\[
	\mbox{RHS } \eqref{eq:1.6} \sim R^{\frac n 2 - \frac 1 {p-1}} \to 0
	\]
as
$R \to \infty$ if $ p < 1 + \frac 2 n$.
It means that for any $u_0 \in L^2(\mathbb R^n) \backslash \{0\}$,
there exists $R_0$ such that \eqref{eq:1.6} is satisfied
with $h(x) = \langle x / R_0 \rangle$.

\subsection{Proof of Proposition \ref{Proposition:1.3}}
The local well-posedness in $H^1(\mathbb R)$ is easily obtained
by the Sobolev embedding and standard contraction argument.
By multiplying \eqref{eq:1.4} by $\overline u$ and $(-\Delta) \overline u$,
integrating over $\mathbb R$,
we obtain
	\begin{align*}
	\frac d {dt} \| u(t)\|_2^2
	&= \|u(t)\|_{p+1}^{p+1}
	\leq C \|u(t)\|_{H^1(\mathbb R)}^{p+1},\\
	\frac d {dt} \| \nabla u(t)\|_2^2
	&= \mathrm{Re}
	\int_{\mathbb R} \nabla (|u(t,x)|^{p-1} u(t,x)) \cdot \overline{\nabla u(t,x)} dx
	\leq C \|u(t)\|_{H^1(\mathbb R)}^{p+1},
	\end{align*}
where $\|f\|_{H^1(\mathbb R)}^2 = \| f \|_2^2 + \| \nabla f \|_2^2$.
By solving the following ordinary differential equality:
	\[
	\frac d {dt} U(t) = C U(t)^{\frac {p+1} 2},
	\]
we get
	\[
	\|u(t)\|_{H^1(\mathbb R)}
	\leq \bigg( \|u_0\|_{H^1(\mathbb R)}^{- (p-1)} - \frac {C(p-1)} 2 t
	\bigg)^{- \frac 1 {p-1}}.
	\]
This proves the Proposition \ref{Proposition:1.3}.




\end{document}